\documentclass[12pt]{article}

\usepackage[T2A]{fontenc}
\usepackage[utf8]{inputenc}

\usepackage{aliascnt}
\usepackage{amsmath, amssymb, amsthm}
\usepackage{hyperref}
\usepackage{color}

\theoremstyle{plain}
\newtheorem{theorem}{Theorem}[section]
\newaliascnt{lemma}{theorem}
\newtheorem{lemma}[lemma]{Lemma}
\aliascntresetthe{lemma}
\newaliascnt{conjecture}{theorem}
\newtheorem{conjecture}[conjecture]{Conjecture}
\aliascntresetthe{conjecture}
\newaliascnt{remark}{theorem}
\newtheorem{remark}[remark]{Remark}
\aliascntresetthe{remark}
\newaliascnt{corollary}{theorem}
\newtheorem{corollary}[corollary]{Corollary}
\aliascntresetthe{corollary}

\begin{document}

\title{Proof of the List Coloring Conjecture for line perfect multigraphs}
\author{Alexey Gordeev\\
Saint Petersburg Department of the Steklov\\
Mathematical Institute\\
\small\tt gordalserg@gmail.com\\}
\date{}

\maketitle

\begin{abstract}We prove that for a line perfect multigraph the chromatic index is equal to the list chromatic index. This is a generalization of Galvin's result on bipartite multigraphs.
	
Soon after the first version was submitted to arxiv, I found out that the same result has already been achieved in an older paper by Peterson and Woodall \cite{PAW}, by a similar but not entirely the same method.
\end{abstract}

\section{Introduction}

All graphs we consider here are finite undirected multigraphs without loops unless said otherwise.

Consider a graph $G$. The \emph{line graph} $L(G)$ is a graph with a vertex for each edge of $G$, where vertices in $L(G)$ are adjacent if and only if the corresponding edges in $G$ are adjacent.

The \emph{chromatic number} $\chi(G)$ of a graph $G$ is the minimum number of colors which can be assigned to vertices in such a way that adjacent vertices have different colors (such an assignment of colors is called \emph{proper coloring}). The \emph{chromatic index} $\chi'(G)$ is the chromatic number of $L(G)$.

Consider a graph $G=(V,E)$ and a function $f:V\to \mathbb{N}$, which assigns a non-negative integer to each vertex. $G$ is \emph{$f$-choosable} if, given sets of colors $\{A_v\}_{v\in V}$, $|A_v|=f(v)$, one can always choose a color for each vertex $\{c_v\}_{v\in V}$, $c_v\in A_v$, so that adjacent vertices have different colors. $G$ is \emph{$n$-choosable} if it is $f$-choosable for a constant function $f(v)=n$. The \emph{choice number} $\chi_l(G)$ is the least $n$ such that $G$ is $n$-choosable. The \emph{list chromatic index} $\chi_l'(G)$ is the choice number of $L(G)$.

One may consider sets $A_v=\{1,\dots,\chi(G)\}$, $v\in V$, to see that $\chi(G)\leq\chi_l(G)$ (and $\chi'(G)\leq \chi_l'(G)$). The \emph{list coloring conjecture} states that for line graphs this inequality turns to equality (see \cite{HAC}, p.$509$ for the history).

\begin{conjecture}[List coloring conjecture (LCC)]
$\chi'(G)=\chi_l'(G)$ for any multigraph $G$.
\end{conjecture}

The LCC has been proven in some special cases. It has been shown to hold for complete graphs of odd degree \cite{HAJ}. The polynomial method was used by different authors to prove the LCC for regular planar multigraphs of class $1$ \cite{EAG} (class $1$ graphs are those whose chromatic index is equal to the maximum degree of a vertex); for complete graphs of degree $p+1$, where $p$ is prime \cite{S}. It was shown to hold asymptotically in some sense \cite{K}. Galvin used the so-called kernel method to prove the LCC for bipartite multigraphs \cite{G}. To achieve it, he used Maffray's characterization of a line perfect graph (\cite{M}, Theorem $1$), a special case of which is a bipartite graph. Naturally, a question arises if Galvin's proof can be extended to line perfect graphs. In this paper, we show that this indeed can be done, though we have to employ some other techniques as well as Galvin's kernel method.

\section{LCC for line perfect multigraphs}

A \emph{clique} is a set of mutually adjacent vertices. An \emph{independent set} is a set of mutually non-adjacent vertices.

Consider a digraph $D=(V,\vec{E})$ with no loops and with at most one arc in each direction between any two vertices. We denote the existence of arc from $v$ to $u$ in $D$ as $v\to u$. The \emph{outdegree} of a vertex $v$ is $d_{out}(v)=|\{u\in V:v\to u\}|$. The \emph{closed neighborhood} of $v$ is $N(v)=\{u\in V:v\to u\}\cup\{v\}$. Clearly, $|N(v)|=d_{out}(v)+1$. The \emph{underlying graph} of $D$ is an undirected graph $G=(V,E)$, where $E=\{\{u,v\}: u\to v \text{ or } v\to u\}$. Note that a pair of opposite arcs corresponds to just one edge in the underlying graph.

A \emph{kernel} of $D$ is an independent set $K\subseteq V$ such that $N(v)\cap K\neq\emptyset$ for any $v\in V$ (or, equivalently, for each $v\in V\setminus K$ there is $u\in K$ such that $v\to u$).

The proof of the next lemma can be found in \cite{G} (see Lemma $2.1$).

\begin{lemma}\label{lem1}
Let $D$ be a digraph in which any induced subgraph has a kernel, $f:V\to \mathbb{N}$ be such that $f(v)>d_{out}(v)$ for any $v$. Then the underlying graph of $D$ is $f$-choosable.
\end{lemma}

An \emph{orientation} of a simple graph $G$ is any digraph whose underlying graph is $G$. A digraph is \emph{normal} if every clique in it has a kernel (which necessarily consists of one vertex). Note that it is sufficient to check that every triangle (a clique on three vertices) has a kernel to prove that a digraph is normal. A graph is \emph{solvable} if every normal orientation of it has a kernel.

\begin{lemma}\label{lem2}
Every induced subgraph of a solvable graph $G=(V,E)$ is solvable.
\end{lemma}
\begin{proof}
It suffices to show that for any $v\in V$ induced subgraph $H$ on $V\setminus v$ is solvable. If we have a normal orientation $D$ of $H$, let us add a vertex $v$ and arcs $\{v\to u:\{v,u\}\in E\}$ to it to get a normal orientation $D'$ of $G$. Since $G$ is solvable, $D'$ has a kernel $K$. Then $K\setminus v$ is a kernel of $D$.
\end{proof}

A graph $G$ is \emph{perfect} if for any induced subgraph $H$ of $G$ the chromatic number $\chi(H)$ is equal to the size of a maximum clique of $H$. Maffray proved that for line graphs the property of being perfect coincides with the property of being solvable (\cite{M}, Theorem $1$).

\begin{theorem}[Maffray]\label{maffray}
The line graph of a multigraph is solvable if and only if it is perfect.
\end{theorem}

In his paper Galvin uses the fact that line graphs of bipartite graphs are perfect. For the line graph of an arbitrary bipartite graph $G$ he finds such an orientation $D$ that all outdegrees in it are smaller than $\chi'(G)$. Since $L(G)$ is perfect, by \autoref{maffray} it is solvable, then by \autoref{lem2} every induced subgraph of $D$ has a kernel. Then by \autoref{lem1} $L(G)$ is $\chi'(G)$-choosable, which proves the LCC for bipartite graphs.

\begin{theorem}[Galvin]
For an arbitrary bipartite multigraph $G$, $\chi'(G)=\chi_l'(G)$.
\end{theorem}

It turns out that the LCC holds for any graph whose line graph is perfect. Such graphs are called \emph{line perfect}.

\begin{theorem}\label{main}
For an arbitrary line perfect multigraph $G$, $\chi'(G)=\chi_l'(G)$.
\end{theorem}

The remainder of the text is devoted to the proof of this theorem.

If $G=(V,E)$, we denote $E(v)=\{e\in E: e \text{ is incident to } v\}$, $d_G(v)=|E(v)|$ for any $v\in V$. We also denote $E(a,b)=E(a)\cap E(b)$, $E(a,b,c)=E(a,b)\cup E(a,c)\cup E(b,c)$ for any distinct $a,b,c\in V$. $E(a,b)$ is the set of all edges between $a$ and $b$, and $E(a,b,c)$ is the set of all edges in the triangle on vertices $a$, $b$, $c$.

\begin{remark}\label{chidef}
For any line perfect multigraph $G$, $\chi'(G)$ is the size of the maximum clique in $L(G)$, so
$$
\chi'(G)=\max\left(\max_{v\in V}(d_G(v)),\max_{a,b,c\in V}(|E(a,b,c)|)\right).
$$
\end{remark}

Let $G=(V,E)$ be an arbitrary line perfect graph. In order to prove that $\chi'(G)=\chi_l'(G)$, we first fix arbitrary sets of colors $\{A_e: |A_e|=\chi'(G)\}_{e\in E}$. Then we choose the order of traversal of a block-cut tree of $G$ by running the depth-first search algorithm, starting from any block. We color biconnected components of $G$ in that order, one after the other.

When we want to color edges of a block $H=(W,F)$, there is at most one vertex $v\in W$ which has some of its incident edges in $E(v)\setminus F(v)$ already colored. If this vertex exists, it is a cut vertex shared with some other block (or blocks) which we have already colored. Let 
$$
C=\{c: \exists\ e\in E(v)\setminus F(v), \text{ color } c \text{ is already assigned to edge } e\}.
$$
For any $e\in F(v)$ we cannot use colors from $C$, so we have to replace $A_e$ with $A_e\setminus C$. Note that $|A_e\setminus C|\geq \chi'(G)-(d_G(v)-d_H(v))\geq d_H(v)$, since $\chi'(G)\geq d_G(v)$. Also, obviously, $\chi'(G)\geq \chi'(H)$.

Given what is stated above, to show that edges of $H$ can be properly colored, it suffices to show that for any $v\in W$, $H$ is $f_v$-edge-choosable, where
$$
f_v(e)= 
\begin{cases}
d_H(v), & \text{if } e \in F(v),  \\
\chi'(H), & \text{otherwise}.
\end{cases}
$$

As we can see, \autoref{main} is equivalent to the following theorem:

\begin{theorem}\label{main2}
Any biconnected line perfect multigraph $G=(V,E)$ is $f_v$-edge-choosable for any $v\in V$, where
\begin{equation*}
f_v(e)= 
\begin{cases}
d_G(v), & \text{if } e \in E(v),  \\
\chi'(G), & \text{otherwise}.
\end{cases}
\end{equation*}
\end{theorem}

You can find a thorough proof of the next theorem in \cite{M} (see Theorem $2$).

\begin{theorem}\label{structure}
Every biconnected component (block) of a line perfect multigraph is either a bipartite graph, $K_4$ (a clique on four vertices), or $K_{1,1,n}$ (a graph consisting of $n+2$ vertices $v_1,\dots,v_n,a,b$ such that $\{v_1,\dots,v_n\}$ is an independent set, and $\{v_i,a,b\}$ is a clique for $i=1,\dots,n$).
\end{theorem}

\begin{remark}
Note that in each case multiple edges are allowed.
\end{remark}

\begin{remark}
Third case is essentially a set of triangles with a common edge, where edges are allowed to be multiple.
\end{remark}

All that is left to do is to prove \autoref{main2} for all three types of blocks described in \autoref{structure}.

For bipartite blocks we build the orientation of the line graph with all outdegrees less than corresponding values of $f_v$, so \autoref{lem1} can be used again. For $K_4$ and $K_{1,1,n}$ we prove $f_v$-edge-choosability directly.

\section{Bipartite blocks}

The orientation we build is essentially the same as the one Galvin used in his proof (see Theorem $4.1$ in \cite{G}).

\begin{theorem} 
Let $G=(V,E)$ be a bipartite multigraph. For any $v\in V$, $G$ is $f_v$-edge-choosable, where
$$
f_v(e)= 
\begin{cases}
d_G(v), & \text{if } e \in E(v),  \\
\chi'(G), & \text{otherwise}.
\end{cases}
$$
\end{theorem}
\begin{proof}
Let $(X,Y)$ be a bipartition of $G$ such that $v\in X$. Let $E(v)=\{e_1,\dots,e_{d_G(v)}\}$. Let $c:E\to\{1,\dots,\chi'(G)\}$ be a proper coloring of $L(G)$. We can assume that $c(e_i)=i$ (it can be achieved by a suitable permutation of colors).

We define the orientation $D$ of $L(G)$ as following: for adjacent edges $e,q\in E(w)$ for some $w\in V$, $e\to q$ if either $w\in X$ and $c(e)<c(q)$, or else $w\in Y$ and $c(e)>c(q)$. Note that for any two multiple edges $e$, $q$ both $e\to q$ and $q\to e$ are true.

For each $e\in E$, $d_{out}(e)<\chi'(G)$, because $c$ is one-on-one on the closed neighborhood $N(e)$ of $e$ in $D$. Additionally, for $e\in E(v)$, $d_{out}(e)<d_G(v)$, because $c$ is still one-on-one on $N(e)$, and $c(q)\leq d_G(v)$ for any $q\in E$, $e\to q$.

Since $G$ is bipartite, any clique in $L(G)$ is the subset of $E(w)$ for some $w\in V$, then by definition of $D$ the kernel consists of the edge with the biggest (if $w\in X$) or the smallest (if $w\in Y$) value of $c$. It follows that $D$ is normal, so we can apply \autoref{lem1}, and $G$ is $f_v$-edge-choosable.
\end{proof}

\section{A transversal case}

Consider a graph $G=(V,E)$ and color sets $\{A_e\}_{e\in E}$. Define
$$
A_F=\bigcup_{e\in F} A_e,\, A(v)=A_{E(v)},\, A(u,v)=A_{E(u,v)}
$$
for $F\subseteq E$, $u,v\in V$.

If there is a pair of non-adjacent edges $e,q\in E$ with intersecting color sets, and $c\in A_e\cap A_q$, we will call a set $\{e,q,c\}$ a \emph{reducing set}. 

We will say that a \emph{transversal case} takes place, if there are no reducing sets, or, equivalently, if $A_e\cap A_q=\emptyset$ for any two non-adjacent edges $e,q\in E$.

In the transversal case, if a proper list edge coloring exists, each edge in it will be assigned a unique color. One can see that finding such a coloring is equivalent to finding a \emph{system of distinct representatives} (which is also called a \emph{transversal}) for a family of finite sets. The next theorem is the reformulation of the famous Hall's marriage theorem (see Theorem $1$ in \cite{H}).

\begin{theorem}\label{hall}
In the transversal case, a proper list edge coloring exists if and only if for any $F\subseteq E$,
$$
|F|\leq |A_F|.
$$
\end{theorem}

In the remaining theorems, our strategy would be to reduce any case to a transversal case, and then to apply \autoref{hall}.

\section{$K_4$ with multiple edges}

The next theorem considers not just cliques but arbitrary graphs on four vertices for the purposes of using the method of induction.

\begin{theorem}\label{thk4}
Let $G=(V,E)$ be a multigraph on four vertices. For any $v\in V$, $G$ is $f_{G,v}$-edge-choosable, where
$$
f_{G,v}(e)= 
\begin{cases}
d_G(v), & \text{if } e \in E(v),  \\
\chi'(G), & \text{otherwise}.
\end{cases}
$$
\end{theorem}
\begin{proof}
Let us fix arbitrary color sets $\{A_e: |A_e|=f_{G,v}(e)\}_{e\in E}$.

We will use the induction on the number of edges in the graph. The base case would be a transversal case. We will consider it at the end of the proof.

If there is a reducing set $\{e,q,c\}$, then let us assign color $c$ to both $e$ and $q$, and consider $G'=(V,E')$, $E'=E\setminus\{e,q\}$, $A'_s=A_s\setminus\{c\}$ for $s\in E'$ (because we cannot assign $c$ to any other edge). It is easy to see that $|E'(w)|=|E(w)|-1$ for any $w\in V$, $|E'(a,b,c)|=|E(a,b,c)|-1$ for any $a,b,c\in V$, so, by \autoref{chidef}, $\chi'(G')=\chi'(G)-1$. Also, $|A'_s|\geq|A_s|-1$ for any $s\in E'$. By induction, $G'$ is $f_{G',v}$-edge-choosable, so we can finish properly coloring the remaining edges using color sets $\{A'_s\}$.

If a transversal case takes place, then in order to apply \autoref{hall} and prove $f_{G,v}$-edge-choosability of $G$, we need to prove that for any $F\subseteq E$ the inequality $|F|\leq |A_F|$ holds.

If $F$ contains a pair of non-adjacent edges $e$, $q$, then $|A_F|\geq |A_e\cup A_q|=|A_e|+|A_q|\geq \chi'(G)+d_G(v)$, since at least one of $e$, $q$ is not incident with $v$. But $\chi'(G)\geq |E(a,b,c)|$, where $\{a,b,c\}=V\setminus \{v\}$, and $E=E(a,b,c)\cup E(v)$, so $|A_F|\geq |E|\geq |F|$.

If $F$ does not contain a pair of non-adjacent edges, then necessarily either $F\subseteq E(w)$ for some $w\in V$, or $F\subseteq E(a,b,c)$ for some $a,b,c\in V$. If $F\subseteq E(v)$, then, taking any $e\in F$, $|A_F|\geq |A_e|=d_G(v)\geq |F|$. In all other cases there exists $e\in F\setminus E(v)$, so $|A_F|\geq |A_e|=\chi'(G)\geq |F|$.

\end{proof}

\begin{corollary}[LCC for graphs on four vertices]
For an arbitrary multigraph $G$ on four vertices, $\chi'(G)=\chi'_l(G)$.
\end{corollary}

\section{$K_{1,1,n}$ with multiple edges}

Same as in the previous section, we consider a more general class of graphs for the purposes of using the method of induction.

Let $G=(V,E)$ be a graph consisting of $n+2$ vertices $v_1,\dots,v_n,a,b$ such that $\{v_1,\dots,v_n\}$ is an independent set. Define $t_G(v_i)=|E(a,b,v_i)|$.

\begin{lemma}\label{degab}
If $e,q\in E$ is a pair of non-adjacent edges, $G'=(V,E')$, where $E'=E\setminus\{e,q\}$, then $d_{G'}(a)=d_G(a)-1$, $d_{G'}(b)=d_G(b)-1$.
\end{lemma}
\begin{proof}
Since $e,q$ are non-adjacent, one of them is incident with $a$, another is incident with $b$, and none of them are incident with both $a$ and $b$, so $d_{G'}(a)=d_G(a)-1$, $d_{G'}(b)=d_G(b)-1$.
\end{proof}

Let us call a vertex $v_i$ \emph{big (in $G$)}, if $t_G(v_i)\geq\max(d_G(a),d_G(b))$, and \emph{great (in $G$)}, if $t_G(v_i)>\max(d_G(a),d_G(b))$.

\begin{lemma}\label{great}
If $v_i$ is great in $G$, then no $v_j$, $j\neq i$ can be big in $G$. If $v_i$ is big in $G$, then no $v_j$, $j\neq i$ can be great in $G$. 
\end{lemma}
\begin{proof}
Suppose $v_i$ is great in $G$. For any $v_j$, $j\neq i$:
$$
|E(a,v_j)|\leq d_G(a)-|E(a,v_i)|-|E(a,b)|<
$$
$$
< t_G(v_i)-|E(a,v_i)|-|E(a,b)|=|E(b,v_i)|.
$$
It follows that
$$
t_G(v_j)=|E(a,b)|+|E(a,v_j)|+|E(b,v_j)|<
$$
$$
<|E(a,b)|+|E(b,v_i)|+|E(b,v_j)|\leq d_G(b),
$$
so $v_j$ is not big in $G$.

If $v_i$ is big in $G$, then by similar reasoning for any $v_j$, $j\neq i$, $|E(a,v_j)|\leq |E(b,v_i)|$ and $t_G(v_j)\leq d_G(b)$, so $v_j$ is not great in $G$.
\end{proof}

\begin{lemma}\label{great2}
Let $t_G(v_1)\geq t_G(v_2)\geq\dots\geq t_G(v_n)$. No $v_i$, $i>1$ can be great in $G$. 
\end{lemma}
\begin{proof}
If $v_i$, $i>1$ is great in $G$, then $v_1$ is also great in $G$, but then by \autoref{great}, $v_i$ is not big in $G$, and we come to the contradiction.
\end{proof}

\begin{lemma}\label{great3}
Let $e,q\in E$ be a pair of non-adjacent edges, $G'=(V,E')$, where $E'=E\setminus\{e,q\}$. If $v_i$ is big (great) in $G$, then $v_i$ is also big (great) in $G'$. If $v_i$ is great in $G'$, then $v_i$ is big in $G$.
\end{lemma}
\begin{proof}
By \autoref{degab}, $d_{G'}(a)=d_G(a)-1$, $d_{G'}(b)=d_G(b)-1$.

If $v_i$ is big in $G$, then $t_{G'}(v_i)\geq t_G(v_i)-1\geq\max(d_G(a),d_G(b))-1=\max(d_{G'}(a),d_{G'}(b))$, so $v_i$ is big in $G'$. By the same reasoning, if $v_i$ is great in $G$, then $v_i$ is great in $G'$.

If $v_i$ is great in $G'$, then $t_G(v_i)\geq t_{G'}(v_i)\geq\max(d_{G'}(a),d_{G'}(b))+1=\max(d_G(a),d_G(b))$, so $v_i$ is big in $G$.
\end{proof}

There are two cases to consider: $v$ from the statement of \autoref{main2} is either one of $\{a,b\}$, or one of $\{v_1,\dots,v_n\}$. We consider each case in a separate theorem.

\begin{theorem}\label{thma}
Let $G=(V,E)$ be a multigraph consisting of $n+2$ vertices $v_1,\dots,v_n,a,b$ such that $\{v_1,\dots,v_n\}$ is an independent set. $G$ is $f_G$-edge-choosable, where
$$
f_G(e)= 
\begin{cases}
d_G(a), & \text{if } e \in E(a),  \\
\max(d_G(a),d_G(b),t_G(v_i)), & \text{if } e \in E(b,v_i).
\end{cases}
$$
\end{theorem}
\begin{remark}
Note that $\max(d_G(a),d_G(b),t_G(v_i))\leq\chi'(G)$, so the statement of the theorem is slightly stronger than needed. This is for the purposes of using the method of induction.
\end{remark}
\begin{proof}	
We can assume that $t_G(v_1)\geq t_G(v_2)\geq\dots\geq t_G(v_n)$.
	
Let us fix arbitrary color sets $\{A_e: |A_e|=f_G(e)\}_{e\in E}$.

We will use the induction on the number of edges in the graph. The base case (a traversal case) would be considered at the end of the proof.

Suppose there is at least one reducing set. If there is such reducing set $\{e,q,c\}$ that $\{e,q\}\cap E(v_1)\neq\emptyset$, then we choose it, otherwise we choose any reducing set. We assign color $c$ to $e$ and $q$, and consider $G'=(V,E')$, $E'=E\setminus\{e,q\}$, $A'_s=A_s\setminus\{c\}$ for $s\in E'$. If we can prove that $|A'_s|\geq f_{G'}(s)$ for any $s\in E'$, then by induction edges of $G'$ can be properly colored using color sets $\{A'_s\}$.

To prove the inequality $|A'_s|\geq f_{G'}(s)$ for $s\in E'$ (from now on we refer to it as the \emph{color set inequality}), it is enough to show that either $f_{G'}(s)=f_G(s)-1$, or that $A'_s=A_s$.

By \autoref{degab}, $d_{G'}(a)=d_G(a)-1$ and $d_{G'}(b)=d_G(b)-1$. This means that for $s\in E'(a)$:
$$
f_{G'}(s)=d_{G'}(a)=d_G(a)-1=f_G(s)-1,
$$
so the color set inequality holds for any $s\in E'(a)$.

By \autoref{great3}, if $v_i$ is not great in $G'$, then it is also not great in $G$, so for any $s\in E'(b,v_i)$: 
$$
f_{G'}(s)=\max(d_{G'}(a),d_{G'}(b))=\max(d_G(a),d_G(b))-1=f_G(s)-1,
$$
and the color set inequality holds for any $s\in E'(b,v_i)$, $v_i$ is not great in $G'$.

Now, we only need to prove the color set inequality for $s\in E'(b,v_i)$, where $v_i$ is great in $G'$. If none of $v_i$ are big in $G$, then none of them are great in $G'$. Otherwise $v_1$ is big in $G$, then by \autoref{great3} it is also big in $G'$, then by \autoref{great} none of $v_i$, $i>1$ are great in $G'$. So, only $v_1$ can be great in $G'$.

The reducing set $\{e,q,c\}$ was chosen in such a way that either $\{e,q\}\cap E(v_1)\neq\emptyset$, or $c\notin A(v_1)$. In the first case $f_{G'}(s)=f_G(s)-1$ for $s\in E'(b,v_1)$, and in the second case $A_s=A'_s$ for $s\in E'(b,v_1)$, so in both cases all color set inequalities hold.

If a transversal case takes place, then in order to apply \autoref{hall} and prove $f_G$-edge-choosability of $G$, we need to prove that for any $F\subseteq E$ the inequality $|F|\leq |A_F|$ holds.

If $F$ does not contain a pair of non-adjacent edges, then necessarily either $F\subseteq E(a)$, $F\subseteq E(b)$ or $F\subseteq E(a,b,v_i)$ for some $i$. In the first case, taking any $e\in F$, $|A_F|\geq|A_e|=d_G(a)\geq|F|$. In all other cases there exists $e\in F\setminus E(a)$, so $|A_F|\geq|A_e|\geq \max(d_G(b), t_G(v_i))\geq|F|$.

If $F$ contains a pair of non-adjacent edges $e$, $q$, then $|A_F|\geq |A_e\cup A_q|=|A_e|+|A_q|\geq d_G(a)+d_G(b)\geq |E|\geq |F|$.
\end{proof}

\begin{theorem}\label{thmv}
Let $G=(V,E)$ be a multigraph consisting of $n+2$ vertices $v_1,\dots,v_n,a,b$ such that $\{v_1,\dots,v_n\}$ is an independent set. $G$ is $f_G$-edge-choosable, where
$$
f_G(e)= 
\begin{cases}
d_G(v_1), & \text{if } e \in E(v_1),\\
\max(d_G(a),d_G(b),t_G(v_1)), & \text{if } e \in E(a,b),\\
\max(d_G(a),d_G(b),t_G(v_i)), & \text{if } e \in E(v_i), i>1.
\end{cases}
$$
\end{theorem}
\begin{proof}
We can assume that $t_G(v_2)\geq t_G(v_3)\geq\dots\geq t_G(v_n)$.

Let us fix arbitrary color sets $\{A_e: |A_e|=f_G(e)\}_{e\in E}$.

We will use the induction on the number of edges in the graph. The base case would be considered at the end of the proof.

If $A(a,b)\not\subseteq A(v_1)$, we take $e\in E(a,b)$ such that there exists $c\in A_e\setminus A(v_1)$, assign color $c$ to $e$ and consider $G'=(V,E')$, $E'=E\setminus\{e\}$, $A'_s=A_s\setminus\{c\}$ for $s\in E'$. $d_{G'}(a)=d_G(a)-1$, $d_{G'}(b)=d_G(b)-1$, and $t_{G'}(v_i)=t_G(v_i)-1$ for all $i$, so $|A'_s|\geq f_{G'}(s)$ for $s\in E'\setminus E'(v_1)$. For $s\in E'(v_1)$, $A'_s=A_s$, so $|A'_s|=|A_s|=f_G(s)=f_{G'}(s)$. Then, by induction, we can finish properly coloring the remaining edges.

From now on we assume that $A(a,b)\subseteq A(v_1)$.

Suppose there is at least one reducing set. We carefully choose some reducing set $\{e,q,c\}$ (below we consider several cases and show the exact way of choosing the reducing set in each case); assing color $c$ to $e$ and $q$, and consider $G'=(V,E')$, $E'=E\setminus\{e,q\}$, $A'_s=A_s\setminus\{c\}$ for $s\in E'$. If we can prove that $|A'_s|\geq f_{G'}(s)$ for any $s\in E'$ (the \emph{color set inequalities}), then by induction edges of $G'$ can be properly colored using color sets $\{A'_s\}$. To prove the color set inequality for $s\in E'$, it is enough to show that either $f_{G'}(s)=f_G(s)-1$, or that $A'_s=A_s$.

By \autoref{degab}, $d_{G'}(a)=d_G(a)-1$ and $d_{G'}(b)=d_G(b)-1$.

If $i>1$ and $v_i$ is not great in $G'$, by \autoref{great3} it is also not great in $G$, so for $s\in E'(v_i)$:
$$
f_{G'}(s)=\max(d_{G'}(a),d_{G'}(b))=\max(d_G(a),d_G(b))-1=f_G(s)-1,
$$
and the color set inequality holds for any choice of $\{e,q,c\}$.

Now, we only need to prove the color set inequality for $s\in E'(v_1)$, $s\in E'(a,b)$, and for $s\in E'(v_i)$, where $i>1$ and $v_i$ is great in $G'$.

Suppose that none of $v_2,\dots,v_n$ are big in $G$, then by \autoref{great3} none of them are great in $G'$. If there is such reducing set $\{e,q,c\}$ that $\{e,q\}\cap E(v_1)\neq\emptyset$, then we choose it, otherwise we choose any reducing set. Note that in second case $c\notin A(v_1)$. In the first case $f_{G'}(s)=f_G(s)-1$ for $s\in E'(v_1)\cup E'(a,b)$, and in the second case $A_s=A'_s$ for $s\in E'(v_1)\cup E'(a,b)$ (here we use the fact that $A(a,b)\subseteq A(v_1)$), so either way the color set inequality holds for any $s\in E'$.

From now on we assume that $v_2$ is big in $G$.

By \autoref{great3}, $v_2$ is also big in $G'$, then by \autoref{great}, $v_3,\dots,v_n$ and $v_1$ are not great in $G'$. That means that for $s\in E'(a,b)$: $f_{G'}(s)=f_G(s)-1$, so the color set inequality holds for $s\in E'(a,b)$ for any choice of the reducing set.

We still need to prove the color set inequality for $s\in E'(v_1)\cup E'(v_2)$.

Suppose we can choose reducing set $\{e,q,c\}$ such that $e\in E(v_1)$ and $q\in E(v_2)$. Then $d_{G'}(v_1)=d_G(v_1)-1$, $t_{G'}(v_2)=t_G(v_2)-1$, so $f_{G'}(s)=f_G(s)-1$ for $s\in E'(v_1)\cup E'(v_2)$, and all color set inequalities hold.

From now on we assume that for any reducing set $\{e,q,c\}$ such that $c\in A(v_1)\cap A(v_2)$, $e$ and $q$ cannot both be from $E(v_1)\cup E(v_2)$. That means that one of the following is true: either $c\notin A(b,v_1)\cup A(b,v_2)$ and $c\in A(a,v_1)\cap A(a,v_2)\cap A(b,v_i)$ for some $i>2$, or  $c\notin A(a,v_1)\cup A(a,v_2)$ and $c\in A(b,v_1)\cap A(b,v_2)\cap A(a,v_i)$ for some $i>2$. We will call color $c$ \emph{$a$-splitting} in the first case, and \emph{$b$-splitting} in the second.

Suppose there are colors $c_1$ and $c_2$ such that $c_1$ is $a$-splitting and $c_2$ is $b$-splitting. Then we can choose two reducing sets $\{e_1,q_1,c_1\}$ and $\{e_2,q_2,c_2\}$ in such a way that $e_1\in E(a,v_1)$, $q_1\in E(b,v_i)$ for some $i>2$, $e_2\in E(b,v_2)$, $q_2\in E(a,v_j)$ for some $j>2$. We assign color $c_1$ to $e_1$ and $q_1$, color $c_2$ to $e_2$ and $q_2$, and consider $G'=(V,E')$, $E'=E\setminus\{e_1,q_1,e_2,q_2\}$, $A'_s=A_s\setminus\{c_1,c_2\}$ for $s\in E'$. To prove that edges of $G'$ can be properly colored using color sets $\{A'_s\}$, we once again need to prove the color set inequalities $|A'_s|\geq f_{G'}(s)$ for all $s\in E'$.

$d_{G'}(a)=d_G(a)-2$, $d_{G'}(b)=d_G(b)-2$. $v_2$ is big in $G$, so by \autoref{great3} (applied twice) it is also big in $G'$, then by \autoref{great} none of $v_3,\dots,v_n$ and $v_1$ are great in $G'$, and the color set inequality holds for $s\in E'(v_i)$, $i>2$, and for $s\in E'(a,b)$.

As for $s\in E'(v_1)$, $c_1\notin A(b,v_1)$ and $c_2\notin A(a,v_1)$, so
$$
|A'_s|\geq |A_s|-1=f_G(s)-1=d_G(v_1)-1=d_{G'}(v_1)=f_{G'}(s).
$$
Similar reasoning shows that $|A'_s|\geq f_{G'}(s)$ for $s\in E'(v_2)$. So, once again, by induction we can finish properly coloring the remaining edges.

From now on we assume that there can only be one type of $a$-splitting, $b$-splitting colors. Because of simmetry we can assume that there are no $b$-splitting colors.

We have reduced an arbitrary case to a case with the following properties: $A(a,b)\subseteq A(v_1)$; either a transversal case takes place, or $v_2$ is big in $G$ and for any reducing set $\{e,q,c\}$ such that $c\in A(v_1)\cap A(v_2)$: $c$ is $a$-splitting. This is the base case of our induction.

The following inequalities hold (essentially we set a weaker bound on the size of $|A_s|$, $s\in E(a,v_2)$, everything else is from definition of $f_G$):
$$
|A_s|\geq 
\begin{cases}
d_G(v_1), & \text{if } s \in E(v_1),\\
\max(d_G(a),d_G(b),t_G(v_1)), & \text{if } s \in E(a,b),\\
\max(d_G(a),d_G(b),t_G(v_i)), & \text{if } s \in E(v_i), i>2,\\
\max(d_G(a),d_G(b),t_G(v_2)), & \text{if } s \in E(b,v_2),\\
\max(d_G(a),d_G(b)), & \text{if } s \in E(a,v_2).
\end{cases}
$$
Additionally,
$$
|A_s\cup A_r|\geq t_G(v_2)+d_G(v_1) \text{ for any } r\in E(b,v_1),s\in E(a,v_2).
$$

We will refer to this set of inequalities as \emph{weak inequalities}.
To prove the existence of proper list edge coloring in the base case, we again employ the method of induction (on the number of edges). We will reduce all cases to a transversal case in such a way that weak inequalities still hold, and then prove that in a transversal case weak inequalities are enough to satisfy conditions of \autoref{hall}.

Suppose there is at least one reducing set. Once again, we carefully choose some reducing set $\{e,q,c\}$ (again, we clarify the exact way of choosing below); assing color $c$ to $e$ and $q$, and consider $G'=(V,E')$, $E'=E\setminus\{e,q\}$, $A'_s=A_s\setminus\{c\}$ for $s\in E'$. If we can prove that weak inequalities still hold for $G'$, $\{A'_s\}$, then by induction edges of $G'$ can be properly colored using color sets $\{A'_s\}$.

By \autoref{degab}, $d_{G'}(a)=d_G(a)-1$, $d_{G'}(b)=d_G(b)-1$, so weak inequalities hold for $s\in E'(a,v_2)$ for any choice of the reducing set.

$v_2$ is big in $G$, so by \autoref{great3} it is also big in $G'$. Also, by \autoref{great} $v_3,\dots,v_n$ and $v_1$ are not great in $G'$. Then weak inequalities also hold for $s\in E'(v_i)$, $i>2$, and for $s\in E'(a,b)$, regardless of the choice of $\{e,q,c\}$.

We still need to prove weak inequalities for $s\in E'(v_1)\cup E'(b,v_2)$, and the last weak inequality.

Suppose we can choose such reducing set $\{e,q,c\}$ that $e\in E(v_1)$ and $q\in E(v_i)$ for some $i>2$. Note that either $c$ is $a$-splitting, or $c\in A(v_1)\setminus A(v_2)$, but in both cases $c\notin A(b,v_2)$. $A_s=A'_s$ for any $s\in E'(b,v_2)$, $|A'_s|\geq|A_s|-1\geq d_G(v_1)-1=d_{G'}(v_1)$ for any $s\in E'(v_1)$, and $|A'_s\cup A'_r|\geq|A_s\cup A_r|-1\geq t_G(v_2)+(d_G(v_1)-1)=t_{G'}(v_2)+d_{G'}(v_1)$ for any $r\in E'(b,v_1)$, $s\in E'(a,v_2)$, so all weak inequalities hold.

From now on we assume that for any reducing set $\{e,q,c\}$: $c\notin A(v_1)$.

Suppose we can choose such reducing set $\{e,q,c\}$ that $c\in A(v_2)\setminus A(v_1)$ and $e\in E(v_2)$.
$$
|A'_s|\geq|A_s|-1\geq t_G(v_2)-1=t_{G'}(v_2)=\max(d_{G'}(a),d_{G'}(b), t_{G'}(v_2))
$$
(because $v_2$ is big in $G'$) for any $s\in E'(b,v_2)$, $A_s=A'_s$ for any $s\in E'(v_1)$, and $|A'_s\cup A'_r|\geq|A_s\cup A_r|-1\geq (t_G(v_2)-1)+d_G(v_1)=t_{G'}(v_2)+d_{G'}(v_1)$ for any $r\in E'(b,v_1)$, $s\in E'(a,v_2)$, so all weak inequalities hold.

From now on we assume that for any reducing set $\{e,q,c\}$: $c\notin A(v_1)\cup A(v_2)$. We choose any such set. $A'_s=A_s$ for any $s\in E'(b,v_2)\cup E'(v_1)$, and $A'_s\cup A'_t=A_s\cup A_t$ for any $r\in E'(b,v_1)$, $s\in E'(a,v_2)$, so all weak inequalities hold.

If a transversal case takes place, then in order to apply \autoref{hall} and prove that edges of $G$ can be properly colored using any color sets $\{A_s\}$, for which weak inequalities hold, we need to prove that for any $F\subseteq E$ the inequality $|F|\leq |A_F|$ holds.

Consider the case when $F$ does not contain a pair of non-adjacent edges. If $F\subseteq E(v_1)$, then taking any $e\in F$, $|A_F|\geq |A_e|\geq d_G(v_1)\geq |F|$. Otherwise $F\subseteq E(a)$, $F\subseteq E(b)$ or $F\subseteq E(a,b,v_i)$ for some $i$.

In all cases except the last one with $i=1,2$, there is $e\in F\setminus E(v_1)$, so $|A_F|\geq|A_e|\geq\max(d_G(a),d_G(b))\geq|F|$ (here we use the fact that by \autoref{great2}, $v_3,\dots,v_n$ are not great in $G$).

In the case $F\subseteq E(a,b,v_1)$, $F\not\subseteq E(v_1)$, so taking any $e\in F\cap E(a,b)\neq\emptyset$, $|A_F|\geq |A_e|\geq t_G(v_1)\geq |F|$.

In the case $F\subseteq E(a,b,v_2)$, $F\not\subseteq E(a)$, so taking any $e\in F\cap E(b,v_2)\neq\emptyset$, $|A_F|\geq |A_e|\geq t_G(v_2)\geq |F|$.

If $F$ contains a pair of non-adjacent edges $e$, $q$, both of which are not in $E(v_1)$, then $|A_F|\geq |A_e\cup A_q|=|A_e|+|A_q|\geq d_G(a)+d_G(b)\geq |E|\geq |F|$.

The remaining case is when $F$ contains some pairs of non-adjacent edges, but in each such pair one of the edges is in $E(v_1)$. Then either $F\setminus E(v_1)\subseteq E(a)$, $F\setminus E(v_1)\subseteq E(b)$, or $F\setminus E(v_1)\subseteq E(a,b,v_i)$ for some $i>1$. In all cases except the last one with $i=2$, we take a pair of non-adjacent edges $e\in F\cap E(v_1)$, $q\in F\setminus E(v_1)$, and $|A_F|\geq|A_e\cup A_q|=|A_e|+|A_q|\geq d_G(v_1)+\max(d_G(a),d_G(b))\geq |F\cap E(v_1)|+|F\setminus E(v_1)|=|F|$ (we again use the fact that by \autoref{great2}, $v_3,\dots,v_n$ are not great in $G$)

In the last case, $F\setminus E(v_1)\subseteq E(a,b,v_2)$ and there is a pair of non-adjacent edges $e\in F\cap E(v_1)$, $q\in F\cap E(v_2)$. If $e\in E(a,v_2)$, $q\in E(b,v_1)$, then $|A_F|\geq |A_e\cup A_q|\geq t_G(v_2)+d_G(v_1)\geq |F|$ (by the last weak inequality). Otherwise $e\in E(b,v_2)$, $q\in E(a,v_1)$, and $|A_F|\geq |A_e\cup A_q|=|A_e|+|A_q|\geq t_G(v_2)+d_G(v_1)\geq |F|$.
\end{proof}

\begin{corollary}
Let $G=(V,E)$ be a multigraph consisting of $n+2$ vertices $v_1,\dots,v_n,a,b$ such that $\{v_1,\dots,v_n\}$ is an independent set, and $\{v_j,a,b\}$ is a clique for $j=1,\dots,n$. For any $v\in V$, $G$ is $f_{G,v}$-edge-choosable, where
$$
f_{G,v}(e)= 
\begin{cases}
d_G(v), & \text{if } e \in E(v),\\
\chi'(G), & \text{otherwise}.
\end{cases}
$$
\end{corollary}
\begin{proof}
If $v=a$ or $v=b$, apply \autoref{thma}. Otherwise, apply \autoref{thmv}.
\end{proof}

This concludes the proof of \autoref{main2}.

\section{Acknowledgments}

I am grateful to Fedor Petrov for introducing me to the LCC, fruitful discussions and useful comments.


\begin{thebibliography}{99}
\bibitem{EAG} M.~N. Ellingham and L. Goddyn. \newblock List edge colourings of some 1-factorable multi-graphs.
\newblock \emph{Combinatorica.} 16:343--352, 1996.	

\bibitem{G} F. Galvin. \newblock The list chromatic index of a bipartite multigraph.
\newblock \emph{J. Combin. Theory Ser. B.} 63:153--158, 1995.

\bibitem{HAC} R. H\"aggkvist and A. Chetwynd. \newblock Some upper bounds on the total and list chromatic numbers of multigraphs.
\newblock \emph{Journal of Graph Theory.} 16(5):503--516, 1992.

\bibitem{HAJ} R. H\"aggkvist  and  J. Janssen. \newblock New bounds on the list-chromatic index of the complete graph and other simple graphs.
\newblock \emph{Combin. Probab. Comput.} 6:295--313, 1997.

\bibitem{H} P. Hall. \newblock On Representatives of Subsets.
\newblock \emph{Journal of the London Mathematical Society.} s1--10:26--30, 1935. 

\bibitem{K} J. Kahn. \newblock Asymptotically good list-colorings.
\newblock \emph{J. Combin. Theory Ser. A.} 73(1):1--59, 1996.

\bibitem{M} F. Maffray. \newblock Kernels in perfect line-graphs.
\newblock \emph{J. Combin. Theory Ser. B.} 55:1--8, 1992.

\bibitem{PAW} D. Peterson and D. R. Woodall. \newblock Edge-choosability in line-perfect multigraphs.
\newblock \emph{Discrete Mathematics.} 202(1-3):191--199, 1999.

\bibitem{S} U. Schauz. \newblock Proof of the list edge coloring conjecture for complete graphs of prime degree.
\newblock \emph{Electron. J. Comb.} 21(3):3--43, 2014.
\end{thebibliography}
\end{document}